\documentclass[a4paper,10pt,reqno]{amsart}
\usepackage{thmtools,enumerate,graphicx,xcolor}
\usepackage{etex,float}
\usepackage{ucs}
\usepackage{amsmath, amssymb, amscd, tabu, diagbox}
\usepackage{amsthm}
\usepackage{sseq}
\usepackage[latin1]{inputenc}
\usepackage{epsfig}
\usepackage{psfrag}

\makeatletter
\def\@tocline#1#2#3#4#5#6#7{\relax
  \ifnum #1>\c@tocdepth 
  \else
    \par \addpenalty\@secpenalty\addvspace{#2}%
    \begingroup \hyphenpenalty\@M
    \@ifempty{#4}{%
      \@tempdima\csname r@tocindent\number#1\endcsname\relax
    }{%
      \@tempdima#4\relax
    }%
    \parindent\z@ \leftskip#3\relax \advance\leftskip\@tempdima\relax
    \rightskip\@pnumwidth plus4em \parfillskip-\@pnumwidth
    #5\leavevmode\hskip-\@tempdima
      \ifcase #1
       \or\or \hskip 2em \or \hskip 2em \else \hskip 3em \fi%
      #6\nobreak\relax
    \hfill\hbox to\@pnumwidth{\@tocpagenum{#7}}\par
    \nobreak
    \endgroup
  \fi}
\makeatother

\usepackage[all]{xy}

\makeatletter
\newcommand{\LeftEqNo}{\let\veqno\@@leqno}
\makeatother
\newlength{\lowerhalftmp}

\usepackage{microtype}

\usepackage{hyperref}
\usepackage[nameinlink]{cleveref}
\crefname{subsection}{subsection}{subsections}
\Crefname{subsection}{Subsection}{Subsections}

\newtheorem{theorem}{Theorem}[section]
\newtheorem{definition}[theorem]{Definition}
\newtheorem{lemma}[theorem]{Lemma}
\newtheorem{prop}[theorem]{Proposition}
\newtheorem{proposition}[theorem]{Proposition}
\newtheorem{corollary}[theorem]{Corollary}

\theoremstyle{definition}
\newtheorem{remark}[theorem]{Remark}

\renewcommand{\epsilon}{\varepsilon}

\DeclareMathAlphabet{\mathpzc}{OT1}{pzc}{m}{it}
\usepackage{mathrsfs}

\newcommand{\Z}{\mathbb{Z}}

\newcommand{\CP}{\mathbb{P}}

\renewcommand{\qed}{$\hfill \square$ \smallskip \\ }

\renewcommand{\bar}{\overline}

\begin{document}
	\thispagestyle{empty}
	\title{The Dold-Whitney theorem and the Sato-Levine invariant}
	\author{Andrew Lobb} 
	\address{Mathematical Sciences\\
		Durham University\\
		UK.}
	\email{andrew.lobb@durham.ac.uk}
	\urladdr{http://www.maths.dur.ac.uk/users/andrew.lobb/}

\begin {abstract} We use the Dold-Whitney theorem classifying $SO(3)$-bundles over a $4$-complex to give a mod 4 obstruction to a 2-component link of trivial linking number being slice.  It turns out that this coincides with the reduction of the Sato-Levine invariant.
\end {abstract}

\maketitle

\section{Introduction}

Let $L$ be a $2$-component link in $S^3$ with trivial linking number.  Choose a Seifert surface for each component of $L$ that misses the other component and such that the surfaces intersect transversely.  The intersection of the two Seifert surfaces gives a framed link in $S^3$.  Such a framed link determines a homotopy class of maps $S^3 \rightarrow S^2$ by the Pontryagin-Thom construction.

\begin{definition}
The Sato-Levine invariant of $L$ is the corresponding group element of $\pi_3(S^2) = \Z$.
\end{definition}

This definition first appears in \cite{Sato}.  The non-vanishing of the Sato-Levine invariant of $L$ provides an obstruction to the link $L$ bounding disjoint locally flat discs in the $4$-ball (in other words, an obstruction to $L$ being slice).

In this paper we give a combinatorially-defined obstruction $\phi(L) \in \Z/4\Z$ to $L$ being slice.  It turns out to be equal to the modulo $4$ reduction of the Sato-Levine invariant.

Nevertheless, the proofs of the well-definedness and properties of $\phi$ are straightforward and direct.  The intermediate construction used in the proofs is a flat $SO(3)$ connection on a $4$-manifold.  The result follows from an application of the Dold-Whitney theorem (which classifies all $SO(3)$ bundles over a $4$-complex by their characteristic classes).

\begin{theorem}[Dold-Whitney \cite{DW}]
	Let $X$ be a 4-dimensional CW-complex.  A principal $SO(3)$ bundle $E$ over $X$ is determined by the pair consisting of its Pontryagin class $p_1(E) \in H^4(X;\Z)$ and second Steifel-Whitney class $w_2(E) \in H^2(X; \Z/2\Z)$.  Furthermore there is an $SO(3)$ bundle $E$ realizing $p_1(E) = a$ and $w_2(E) = b$ exactly when
	\[ \bar{a} = b^2 \in H^4(X; \Z/4\Z) \]
	where we write $\bar{a}$ for the reduction of $a$ and where the squaring of $b$ is the Pontryagin squaring operation.
\end{theorem}

In essence, we are giving an essentially 4-dimensional proof of the invariance and properties of a reduction of the Sato-Levine invariant.

\subsection*{Acknowledgements}
We thank the Max Planck Institute for their hospitality and thank the colleagues there who showed such an interest in this interpretation of a well-known invariant.

\section{Definition and properties}

Let $L$ be an oriented link in $S^3$ of trivial linking number comprising two components $K_1$ and $K_2$.  Then there certainly exist two disjoint locally flat immersed discs in the $4$-ball $B^4$, bounded by $L$, where the discs are boundary-transverse and oriented consistently with $L$.  Let $D_1$ and $D_2$ be two such discs.

\begin{definition}
\label{defn:whatitis}
To each self-intersection point $p \in B^4$ of $D_1$ or $D_2$ we associate a number $i(p) \in \{ -1, 0 , 1 \}$ as follows.

Let $\{ s,t \} = \{ 1,2 \}$, and suppose that $p$ is a self-intersection point of $D_s$.  Choose a loop $l$ which starts and ends at $p \in B^4$, staying on $D_s$ and starting and ending on different branches of the intersection.  Then we set
\[ w(p) : = [l] \in H_1(B^4 \setminus D_t ; \Z/2\Z) = \Z/2\Z = \{ 0 , 1 \} {\rm .} \]
Note that this is independent of the choice of $l$.

We define
\[ i(p) = w(p) \sigma(p) \]
where $\sigma(p) = \pm 1$ is the sign of the intersection at $p$.
\end{definition}

\begin{definition}
We define

\[ \phi(L, D_1, D_2) = \sum_p i(p) \in \Z/4\Z \]

\noindent where the sum is taken over all the self-intersections $p$ of $D_1$ and $D_2$.
\end{definition}

\begin{remark}
	The fact that $\phi$ is the reduction of the Sato-Levine invariant may be deduced from this definition and the crossing-change formula due to Jin \cite{Jin} and Saito \cite{Saito}.
\end{remark}

We shall show the following
\begin{proposition}
	\label{prop:gluing}
	Suppose that $L$ bounds the two pairs of disjoint locally flat immersed discs $(D_1, D_2)$ and $(D'_1, D'_2)$.  Then there exists a closed $4$-manifold $X$ with a flat $SO(3)$-bundle $E \rightarrow X$ with
	\[ \phi(L, D_1, D_2) - \phi(L,D'_1,D'_2) = w_2^2(E) = p_1(E) = 0 \in \Z / 4\Z = H^4(X ; \Z / 4\Z) \rm{.} \]
\end{proposition}

From this proposition we immediately obtain a corollary.

\begin{corollary}
\label{cor:obstruction}
The quantity $\phi(L, D_1, D_2)$ depends only on the link $L$.  So we can write $\phi(L) = \phi(L,D_1,D_2)$.  Furthermore, if $\phi(L) \not= 0$ then $L$ does not bound two disjoint embedded locally flat discs in $B^4$. \qed
\end{corollary}
We note that the content of the equation in Proposition \ref{prop:gluing} is the first equality sign, the second being the Dold-Whitney theorem (the squaring operation here is the Pontryagin square, a $\Z/4\Z$ lift of the cup product), and the third being a consequence of the flatness of the bundle $E$.

\begin{remark}
	\label{rem:saito}
	Work by Saito \cite{Saito} gives a $\Z/4\Z$-valued extension of the Sato-Levine invariant for links of even linking number.  Saito's invariant is constructed via considering the framed intersection of possibly non-orientable Seifert surfaces, and is distinct from that which we consider.
\end{remark}

We devote the following section to the description of the manifold $X$ and the $SO(3)$-bundle $E \rightarrow X$.

\section{Construction of a 4-manifold with an $SO(3)$-bundle}

Given an immersed locally-flat $2$-link $\Lambda \subseteq S^4$ of two components with no intersections between distinct components of the link, we give a construction of a closed diagonal $4$-manifold $X_\Lambda$.

Suppose that $\Lambda$ has $n_-$ negative and $n_+$ positive intersection points.  Then we blow-up each negative intersection point by taking connect sum with $\overline{\CP}^2$ and each positive intersection point by taking connect sum with $\CP^2$.  Let
\[ \overline{\Lambda} \hookrightarrow n_-\bar{\CP}^2 \# n_+\CP^2 \]
be the proper transform of $\Lambda$.

Because of the way we chose to blow-up the negative and positive intersections respectively, each exceptional sphere intersects $\bar{\Lambda}$ in two points, once negatively, and once positively.  Furthermore, since the self-intersections of $\Lambda$ do not occur between the distinct components of $\Lambda$, each exceptional sphere intersects exactly one component of $\bar{\Lambda}$.

This means that each component of $\bar{\Lambda}$ is trivial homologically, and so has a trivial $D^2$-neighborhood.  This allows us to do surgery by removing a neighborhood $\bar{\Lambda} \times D^2$ and gluing in two copies of $D^3 \times S^1$.  We call the resulting manifold $X_\Lambda$.  Now we collect some information about the algebraic topology of $X_\Lambda$.

\begin{prop}
\label{prop:algtop}
The 4-manifold $X_\Lambda$ has diagonal intersection form and satisfies
\begin{align*}
H_1(X_\Lambda ; \Z) = \Z^2, \,\,\,& H_2(X_\Lambda ; \Z) = \Z^{n_+ + n_-},\\
b_2^+ = n_+,\,\,\,& b_2^- = n_- \rm{.}
\end{align*}
\end{prop}

\begin{proof}
We shall display $n_- + n_+$ disjoint embedded tori in $X_\Lambda$, $n_-$ of which have self-intersection $-1$ and $n_+$ of which have self-intersection $+1$.  Using a simple argument counting handles and computing Euler characteristics, it is easy then to deduce the statement of the proposition.

Each exceptional sphere $E \subset n_-\bar{\CP}^2 \# n_+\CP^2$  intersects $\bar{\Lambda}$ transversely in two points. 
Connect these two points by a path on $\bar{\Lambda}$.  The $D^2$-neighborhood of $\bar{\Lambda}$ pulls back to a trivial
$D^2$-bundle over the path.  The fibers over the two endpoints can be identified with neighborhoods on $E$.
Removing these neighborhoods from $E$ we get a sphere with two discs removed and we take the union of this
with the $S^1$ boundaries of the all the fibers of the $D^2$-bundle over the path.

This either gives a torus or a Klein bottle.  Because $E$ intersects $\bar{\Lambda}$ once positively and once negatively, we see that we in fact get a torus which has self-intersection $\pm 1$.  Finally note that we can certainly choose paths on $\bar{\Lambda}$ for each exceptional sphere which are disjoint.
\end{proof}

\section{A flat connection and the Dold-Whitney theorem}

This section considers the characteristic classes of $SO(3)$-bundles, but in fact we shall only be concerned with those bundles whose structure group can be restricted to a small subgroup of $SO(3)$.

\begin{definition}
	\label{def:klein4groupthing}
Let $V_4  \subseteq SO(3)$ be the Klein $4$-group
\[ V_4 = \left\{ \left( \begin{array}{ccc} 1 & 0 & 0 \\ 0 & 1 & 0 \\ 0 & 0 & 1 \end{array} \right), \left( \begin{array}{ccc} 1 & 0 & 0 \\ 0 & -1 & 0 \\ 0 & 0 & -1 \end{array} \right), \left( \begin{array}{ccc} -1 & 0 & 0 \\ 0 & 1 & 0 \\ 0 & 0 & -1 \end{array} \right), \left( \begin{array}{ccc} -1 & 0 & 0 \\ 0 & -1 & 0 \\ 0 & 0 & 1 \end{array} \right) \right\} { \rm .}\]
In future, we write $x_1, x_2, x_3$ for the non-identity elements.
\end{definition}

We begin with a well-known (in certain circles) lemma about a flat $SO(3)$-connection on the torus.

\begin{lemma}
	\label{lem:flat_connection_torus}
	Let $T^2$ be a torus and let $\eta : \pi_1(T^2) \rightarrow SO(3)$ be defined by $\eta(a) = x_1$ and $\eta(b) = x_2$ where $a, b$ is a basis for $\pi_1(T^2) = H_1(T^2; \Z) = \Z \oplus \Z$.  Writing $E_\eta$ for the associated (flat) $SO(3)$-bundle, we have
	\[ w_2 (E_\eta) = 1 \in H^2(T^2; \Z/2) = \Z/2 {\rm .} \]
\end{lemma}

\begin{proof}
	Note that the matrices of $V_4$ are all diagonal with entries in $\Z/2\Z = O(1)$.  Hence, thinking of $E_\eta$ as an $O(3)$-bundle, we can write $E_\eta = L_1 \oplus L_2 \oplus L_3$ where $L_i$ is the (flat) real line bundle determined by the representation
	\[ \pi_1(T^2) \stackrel{\eta}{\longrightarrow} V_4 \stackrel{p_i}{\longrightarrow} \Z/2\Z = O(1) {\rm ,} \]
	where $p_i$ is given by the $(ii)$ matrix entry.
	
	Each $L_i$ is the pullback of a M\"obius line bundle over a circle by a map $T^2 \rightarrow S^1$ (depending on $i$) which is a projection map onto an $S^1$ factor of $T^2$.  We compute then that
	\[ w_1(L_1) = \bar{a}, \, \, w_1(L_2) = \bar{b}, \, \,  {\rm and} \, \, \, w_1(L_3) = \bar{a} + \bar{b} {\rm ,} \]
	where we write $\bar{a}, \bar{b} \in H^1(T^2; \Z/2\Z)$ for the reductions of the Poincar\'e duals of $a$ and $b$ respectively.
	
	Then we compute via the cup-product formula for the Stiefel-Whitney class of a sum of bundles:
	\[ w_2(E_\eta) = \bar{a}\cup\bar{b} + \bar{b}\cup(\bar{a} + \bar{b}) + (\bar{a} + \bar{b})\cup\bar{a} = \bar{a} \cup \bar{b} = 1 \in H^2(T^2 ; \Z/2\Z) {\rm .}  \]
\end{proof}

\begin{remark}
	\label{rem:surjectivity_torus}
	For representations $\eta : \pi_1(T^2) \rightarrow V_4$, Lemma \ref{lem:flat_connection_torus} says that $w_2(E_\eta)$ is non-trivial exactly when $\eta$ is surjective (note that if $\eta$ is not surjective then $E_\eta$ is the pullback of a bundle over a circle).
\end{remark}

Suppose now that we are in the situation of the hypotheses of Proposition \ref{prop:gluing}.  By gluing together the two pairs of disks $(D_1, D_2)$ and $(D'_1, D'_2)$ along their boundary $L \subset S^3$, we get a $2$-component locally-flat immersed link $\Lambda \subset S^4$.  We write $\Lambda_j$ for the sphere resulting from gluing together $D_j$ and $D'_j$ for $j = 1,2$.  In performing this gluing we of course reverse the orientation of the second 4-ball.  This has the effect that positive/negative self-intersections of $(D'_1, D'_2)$ become negative/positive self-intersections of $\Lambda$ respectively.  We write $X = X_\Lambda$, and now give a flat $SO(3)$ connection on $X$.

Let $\theta : \pi_1 (X) \rightarrow SO(3)$ be a representation that factors through an onto map $\bar{\theta}: H_1(X ; \Z) = \Z \oplus \Z \rightarrow V_4$.  We define $\theta$ by setting $\bar{\theta} : m_j \mapsto x_j$ where $m_j$ is a meridian of $\Lambda_j$ for $j = 1,2$.  We write $E_\theta$ for the associated (flat) $SO(3)$-bundle over $X$.
We are interested in the characteristic classes $w_2(E_\theta) \in H^2(X;\Z/2\Z)$ and $p_1(E_\theta) \in H^4(X;\Z)$.
In the case we consider in this paper, we know immediately that $p_1(E_\theta) = 0$ since the bundle admits a flat connection.

Proposition \ref{prop:gluing} now follows by computing $w_2^2(E_\theta)$ using our basis of tori representing the second homology of $X$.

\begin{proof}[Proof of Proposition \ref{prop:gluing}]
As noted before, the content of the proposition is in the first equality sign, namely that we have
\[ w_2^2(E_\theta) = \phi(L, D_1, D_2) - \phi(L, D'_1, D'_2) \in H^4(X ; \Z/4\Z) {\rm .} \]

We compute $w_2(E_\theta) \in H^2(X ; \Z/2\Z)$ by pulling back the representation $\theta$ to each torus representing a basis element of $H_2(X ; \Z)$.
Let $T_p \subseteq X$ be a torus as constructed in Proposition \ref{prop:algtop} coming from a self-intersection point $p \in \Lambda_j$ for some $j \in \{1,2\}$.  We wish to give a pair of $H_1(T_p ; \Z)$-generating circles on $T_p$.

The first of these circles we take to be a meridian $m_p$ to $\Lambda_j$.  The other we take to be any circle $l_p$ on $T_p$ which is dual to $m_p$.  Then the restriction of $\theta$ to $\pi_1(T_p) = H_1(T_p ; \Z)$ is determined by $\bar{\theta}(m_p)$ and $\bar{\theta}(l_p)$.

We know by the definition of $\theta$ that we have $\bar{\theta}(m_p) = x_j$.  On the other hand, $\bar{\theta}(l_p)$ is determined by the class of $l_p$ in $H_1(X ; \Z/2\Z)$.  Consider $w(p)$ as given in Definition \ref{defn:whatitis}.  If we have $w(p) = 0$ then $\bar{\theta}(l_p) \in \{ 1, x_j \}$, but if $w(p) = 1$ then $\bar{\theta}(l_p) \notin \{ 1, x_j \}$.  In consequence, $\theta \vert_{\pi_1(T_p)}$ maps onto $V_4$ if and only if $w(p) = 1$.

In light of Remark \ref{rem:surjectivity_torus}, it follows that $w_2(E_\theta \vert_{T_p}) = w(p) \in \Z/2\Z = H^2(T, \Z/2\Z)$.

The equation we wish to prove then follows since, computing in $H^4(X, \Z/4\Z)$, we have
\begin{align*}
p_1(E_\theta) &= w_2^2(E_\theta) = \left( \sum_p (w_2(E_\theta)[T_p])\overline{[T_p]} \right)^2 \\
&= \sum_p (w_2(E_\theta \vert_{T_p})[T_p])(\overline{[T_p]} \cup \overline{[T_p]}) = \sum_p w(p)(\overline{[T_p]} \cup \overline{[T_p]}) \\
&= \phi(L, D_1, D_2) - \phi(L,D'_1,D'_2) {\rm ,}
\end{align*}
where we write $[T_p]$ for the fundamental class of $T_p$ and the overline denotes the Poincar\'e dual.  We use here that the Pontryagin square of the $\Z/2\Z$ reduction of an integral class is the $\Z/4\Z$ reduction of the usual square of that integral class.
\end{proof}

\begin{remark}
It is possible to give more a complicated construction along the lines above, which should extend the invariant to $2$-component links of even linking number.  This recovers the $\Z/4\Z$ reduction of the Sato-Levine invariant due to Akhmetiev and Repovs \cite{ZhRe} for this class of links.

The construction above starts with two pairs of discs $(D_1, D_2)$ and $(D'_1, D'_2)$.  In the case of a link $L$ of non-zero linking number $2n$ we start rather with two immersed concordances from $L$ to the $(2,4n)$-torus link.  These may then be glued end-to-end and the resulting immersed surface resolved by blow-up in order to give two embedded tori $\Lambda$ of self-intersection $0$ in a blow-up of $S^1 \times S^3$.  Surgery may be done on $\Lambda$ in order to give a closed $4$-manifold $X$.

The main subtleties in this new situation are in performing the surgery so that one obtains $X$ with the correct algebraic topology, and in dealing with an intersection form that is no longer diagonal.
\end{remark}

\bibliographystyle{amsplain}
\bibliography{References}

\providecommand{\bysame}{\leavevmode\hbox to3em{\hrulefill}\thinspace}
\providecommand{\MR}{\relax\ifhmode\unskip\space\fi MR }
\providecommand{\MRhref}[2]{%
  \href{http://www.ams.org/mathscinet-getitem?mr=#1}{#2}
}
\providecommand{\href}[2]{#2}
\begin{thebibliography}{1}

\bibitem{ZhRe}
P.M. Akhmetiev and D.~Repovs, \emph{A generalization of the {S}ato-{L}evine
  invariant}, Proc. Steklov Inst. Math. \textbf{221} (1998), 60--70.

\bibitem{DW}
A.~Dold and H.~Whitney, \emph{Classification of oriented sphere bundles over a
  4-complex}, Ann. Math. \textbf{69} (1959), no.~3, 667--677.

\bibitem{Jin}
G.~T. Jin, \emph{A calculation of the {S}ato-{L}evine invariant}, preprint,
  Brandeis.

\bibitem{Saito}
M.~Saito, \emph{On the unoriented {S}ato-{L}evine invariant}, J. Knot Theory
  and Ramif. \textbf{2}, no.~3, 335--358.

\bibitem{Sato}
N.~Sato, \emph{Cobordism of semi-boundary links}, Topol. and Appl. \textbf{18}
  (1984), 225--231.

\end{thebibliography}

\end{document}